\documentclass[10pt,twocolumn,twoside]{IEEEtran}

\usepackage{amsmath,amssymb,hyperref,latexsym,amsthm}
\newtheorem{theorem}{Theorem}[section]
\newtheorem{lemma}[theorem]{Lemma}

\usepackage{xcolor}
\hypersetup{colorlinks}
\usepackage{graphicx}
\usepackage{epsf}
\usepackage{amsfonts}
\usepackage{graphics}
\usepackage{epsfig}
\usepackage{float}
\usepackage{extarrows}

\usepackage{enumerate}
\usepackage{tikz-cd}

\usepackage[ruled,vlined,linesnumbered]{algorithm2e}

 \newtheorem{assu}[theorem]{Assumption}

 \numberwithin{equation}{section}

\newcommand{\beq}{\begin{equation}}
\newcommand{\eeq}{\end{equation}}

\newcommand{\bR}{\mathbb{R}}
\newcommand{\bN}{\mathbb{N}}

\newcommand{\R}{\bR}
\newcommand{\Rn}{\bR^n}

\newcommand{\cJ}{\mathcal{J}}
\newcommand{\cN}{\mathcal{N}}

\newcommand{\argmin}{\mathop{\mathrm{argmin}}}

\newcommand{\norm}[1]{\left\Vert #1\right\Vert}

\newcommand{\znorm}[1]{\norm{#1}_0}

\newcommand{\ip}[2]{\left\langle #1,\, #2\right\rangle}
\newcommand{\set}[2]{\left\{#1\,\left|\, #2\right.\right\}}
\newcommand{\map}[3]{#1:\, #2\rightarrow #3}

\newcommand{\gam}{\gamma}

\newcommand{\lam}{\lambda}

\newcommand{\eps}{\epsilon}
\newcommand{ \vareps}{\varepsilon}

\newcommand{\sig}{\sigma}
\newcommand{\alf}{\alpha}
\newcommand{\del}{\delta}

\newcommand{\bx}{\bar x}

\newcommand{\tx}{{\tilde x}}
\newcommand{\te}{{\tilde e}}

\newcommand{\hn}{{\hat n}}

\usepackage{xcolor}
\definecolor{newcolor}{rgb}{.8,.349,.1}

\newcommand{\red}[1]{\textcolor{red}{#1}}

\title{IRLS for Sparse Recovery Revisited: \\ Examples of Failure and a Remedy}

\author{Aleksandr Aravkin, James Burke, Daiwei He}

\date{\today}

\begin{document}

\maketitle

\begin{abstract}
Compressed sensing is a central topic in signal processing with myriad applications, where the goal is to recover a signal
from as few observations as possible. Iterative re-weighting is one of the fundamental tools
to achieve this goal.
This paper re-examines 
the iteratively reweighted least squares (IRLS) algorithm for sparse recovery proposed by
Daubechies, Devore, Fornasier, and G\"unt\"urk in 
\emph{Iteratively reweighted least squares minimization for sparse recovery},
{\sf Communications on Pure and Applied Mathematics}, {\bf 63}(2010) 1--38.
Under the null space property of order $K$, 
the authors show that 
their algorithm converges to the unique $k$-sparse solution for $k$ strictly bounded above by a value strictly less than $K$,
and this $k$-sparse solution coincides with the unique $\ell_1$ solution. 
On the other hand, it is known that, for $k$ less than or equal to $K$, the $k$-sparse and $\ell_1$ solutions
are unique and coincide. 
The authors emphasize that their proof method
does not apply for $k$ sufficiently close to $K$,
and remark that they were unsuccessful
in finding an example where the algorithm fails for these values of $k$. 

In this note we construct a family of examples where the Daubechies-Devore-Fornasier-G\"unt\"urk
IRLS algorithm fails for 
$k=K$, 
and provide a modification to their algorithm that provably
converges to the unique $k$-sparse solution for $k$ less than or equal to $K$ while preserving the local linear rate. 
The paper includes numerical studies of this family as well as the modified IRLS algorithm, 
testing their robustness under perturbations and to parameter selection.
\end{abstract}

\section{Introduction}
The fundamental problem in compressed sensing is to recover 
 the sparsest solution $x_*$ to a linear equation of the form $\Phi x=y$ for a given $y$,
 where $\Phi \in \mathbb{R}^{\ell \times N}$ is the measurement matrix and $\ell < N$. 
 We denote the set of solutions to the equation $\Phi x=y$ by $\Phi^{-1}(y)$ 
 which is assumed to be non-empty throughout.
 The problem of obtaining the sparsest solution can be posed as the minimization
 of the so-called $0$-norm, $\znorm{x}$, over $\Phi^{-1}(y)$, where 
 $\znorm{x}$ is the number of non-zero components in the vector $x$.
 Since the $0$-norm problem is NP hard, in practice \cite{chen2001atomic}
 one replaces this problem
 with the 
$\ell_1$ minimization (or basis pursuit) problem 
\begin{equation}
	\min_{x \in \Phi^{-1}(y)} \norm{x}_1 \label{1}\tag{BP}.
\end{equation}
The relationship of \ref{1} to the $0$-norm problem has been intensively studied 
over the past few years 
\cite{candes2004robust,candes2005decoding,  donoho2006compressed, donoho2003optimally}.
Compressed sensing has applications to a range of signal processing areas, 
including image acquisition, sensor networks and image reconstruction 
\cite{chen2001atomic, lustig2005application, takhar2006new}.

Numerous algorithms have been proposed for solving \ref{1} and its various reformulations, which include the basis pursuit denoising (BPDN) problem\red{:}
$$\min_x \{\norm{x}_1 | \norm{\Phi x - y}_2 \leq \sigma\},$$ 
the LASSO problem\red{:}
$\min_x \norm{x}_1 + \frac{\mu}{2} \norm{\Phi x - y}_2^2,$
and  the $\ell_1$-regression problem\red{:}
\begin{equation}\label{pr:l1}\tag{$\ell_1$R}
\min_x \norm{Az - b}_1
\end{equation}
under the correspondences $\text{rge}\,(A) = \text{Null}(\Phi)$ and $\Phi b = y$ \cite{candes2005decoding} (see Section \ref{sec:exp} for details). 
Algorithms designed to solve these problems include
the iteratively reweighted least squares (IRLS) algorithms
\cite{BCWW2015,Lin2009} which apply to $\ell_1$R, the
FISTA algorithm \cite{beck2009fast, wright2009sparse} which applies to the LASSO, 
and the homotopy algorithm \cite{osborne2000new}, the alternating direction method of multipliers (ADMM) \cite{boyd2011distributed, gabay1975dual}, and the level-set method described in  
\cite{ABDFR19} which all apply to BPDN. However, the focus of this paper is
the IRLS algorithm described in \cite{daubechies2010iteratively} which we refer to as the DDFG-IRLS algorithm. 

In \cite{daubechies2010iteratively}, the authors show that if the matrix $\Phi$ satisfies the the \emph{null space property
of order $K$ for $0<\gam<1$} (see Section \ref{sec:conv} for details), then the DDFG-IRLS algorithm 
converges to the unique $k$-sparse solution when $k< K-2\gam(1-\gam)^{-1}$,
and this $k$-sparse solution coincides with the unique $\ell_1$ solution, where a
vector is $k$-sparse if it has $k$ nonzero components. 
In addition, the authors also establish the local linear convergence of the DDFG-IRLS algorithm
when $0<\gam<1-2/(K+2)$.
On the other hand, it is known that for $k\le K$ the $k$-sparse and $\ell_1$ solutions
are unique and coincide \cite{gribonval2003sparse, donoho1989uncertainty, daubechies2010iteratively}.
In \cite[Remark 5.4]{daubechies2010iteratively}, the authors note that
their proof method
does not apply for $K-2\gam(1-\gam)^{-1}\le k \le K$, and state that
they were unsuccessful
in finding an example where the algorithm fails when $k$ falls in this range.
In this note we construct a family of examples where the DDFG-IRLS algorithm
fails when $k=K$, and provide a modification to their algorithm that provably
converges to the unique $k$-sparse solution for $k\le K$. In addition, we show that this modification
is locally linearly convergent for all $k\le K$ and $\gam\in (0,1)$ which increases the range
of $\gam$ values for which linear convergence is assured.

Iteratively re-weighted least squares algorithms (IRLS) for solving $\ell_p$ minimization problems for $1\le p\le \infty$
have been in the literature for many years beginning with the Ph.D. thesis of Lawson \cite{lawson1961contribution}.
For $0<p\le1$, 
IRLS was used to solve sparse reconstruction in \cite{gorodnitsky1997sparse}, 
and a theory for solving $\ell_p$ minimization problems in general can be found in \cite{osborne1985finite}. 
We refer the reader to \cite{oleary1990} for a survey on IRLS methods applied to robust regression.
More recently, cluster point convergence of IRLS smoothing methods for problems of the form $\min f(x)+\lam \norm{x}_0$, where 
$\map{f}{\Rn}{\R\cup\{+\infty\}}$, is given in \cite{Lin2009}. 
In addition, an IRLS algorithm has been developed
for convex inclusions of the form $A_ix+b_i\in C_i,\ i=1,\dots,n$ where the sets $C_i$ are all assumed to be convex
\cite{BCWW2015}. In this case, the authors establish the iteration complexity of their method. 
However, all of these methods focus on
general linear systems and do not specifically address the problem of compressed sensing where the  
null space properties play a key role. 
Daubechies, Devore, Fornasier, and G\"unt\"urk \cite{daubechies2010iteratively}
focus on the compressed sensing case where $\norm{x}_0$ is approximated by a smoothing of the
norms $\norm{x}_p$ for $0<p\le1$. We follow 
Daubechies, Devore, Fornasier, and G\"unt\"urk in the $p=1$ case
and suggest a simple modification to their method for updating the smoothing parameter. This modification allows us to
obtain stronger convergence properties.

Our discussion proceeds as follows. In Section 2 we discuss the DDFG-IRLS algorithm
and our modification to the smoothing parameter update procedure. In Section 3 we prove the stronger
convergence and rate of convergence properties for the modified algorithm. 
Our proofs closely parallel those given in \cite{daubechies2010iteratively}
but contain some simplifications. In Section 4, we construct a family of examples where the 
DDFG-IRLS algorithm \red{fails} but our modifications succeed. These results are illustrated numerically in Section 5
where
we also provide a few numerical experiments to illustrate the numerical stability of the modified algorithm.
In particular, we show that on randomly chosen problems the two methods have virtually identical
performance characteristics.

\section{The Modified IRLS Algorithm}
Our algorithm is similar to the IRLS algorithm given in 
\cite{daubechies2010iteratively}. 
The primary innovation is the manner in
which the smoothing parameter $\eps_k$ is is updated. In \cite{daubechies2010iteratively},
$\eps_k$ is updated by the rule 
\[
\epsilon_{k+1} = \min\left\{\epsilon_k, \frac{r_{K+1}(x^{k+1})}{N}\right\},
\]
where, for $x = (x_1, ..., x_N)^T\in\R^N$, 
\[
\mbox{$r_i(x)$ is the $i$th largest element of 
$\{|x_j|| 1 \leq j \leq N\}$. }
\]
On the other hand, the algorithm below employs the update rule
\begin{equation}\label{eq:eps update}
\epsilon_{k+1} = \min\left\{\epsilon_k, \frac{\eta(1-\gamma) \sigma_K(x^{k+1})}{N}\right\},
\end{equation}
where $\eta\in(0,1)$ is chosen and fixed at the beginning of the iteration,
the parameters $\gam$ and $K$ come from A2, and 
\begin{equation}\label{eq:sig}
\sigma_j(z) := \sum_{\nu > j}r_\nu(z),\ \ j=1,\dots,N.
\end{equation}

As stated, the algorithm is an iteratively re-weighted least squares algorithm 
where the weights at each iteration are given by
\begin{equation}\label{eq:weights}
w^k_i := ((x_i^k)^2 + \eps_k^2)^{-1/2}\quad i=1,\dots,N.
\end{equation}
Moreover, given a positive weight vector $w\in\R^N_{++}$, we define the 
associated inner product by
\[
\ip{u}{v}_w := \sum_{i=1}^N w_i u_i v_i\quad\forall\, u,v\in\R^N,
\]
and the corresponding weighted 2-norm by 
$\norm{u}_w := \sqrt{\ip{u}{u}_w}$. With this notation, our algorithm can be stated
as follows. 
\bigskip



\begin{algorithm}[h!]\label{al:1}
\DontPrintSemicolon
\SetAlgoLined
\SetKwInOut{Input}{Input}\SetKwInOut{Output}{Output}
\Input{$x^0\in\R^N$ \\ Initialize $\eps_0 = 1$ and $\eta \in (0, 1)$}
\BlankLine
 
\While{not converge}{
$w^k_i \quad\! \leftarrow ((x_i^k)^2 + \eps_k^2)^{-1/2}\quad i=1,\dots,N.$\;
    $x^{k+1} \leftarrow \argmin \set{\norm{x}^2_{w^k}}{x \in \Phi^{-1}(y)}.$\;
%
    $\epsilon_{k+1} \ \leftarrow \min\left\{\epsilon_k, \frac{\eta(1-\gamma) \sigma_K(x^{k+1})}{N}\right\}$.\;
    If $\epsilon_{k+1} = 0$, stop.\;
    $k \leftarrow k+1$.
}
\Output{$x^{k+1}$}
\caption{An IRLS algorithm for compressed sensing.} 
\label{algorithm_cs}
\end{algorithm}
\medskip

In general, the null space parameters $K$ and $\gamma$ are unknown, however, we show in Section \ref{sec:num} that the performance of both algorithms is robust with respect to their choice.
In particular, by taking $K=N/2$ and $\gam=.9$, the algorithms DDFG-IRLS and Algorithm \ref{algorithm_cs}
perform essentially the same in successfully solving the BP problem.

\section{Convergence}\label{sec:conv}
We follow the proof strategy given in \cite{daubechies2010iteratively} for establishing the convergence
and rate of convergence of Algorithm \ref{algorithm_cs}.
Given $\eps>0$, consider the smoothed $\ell_1$ objective
\[
J(x, \eps) := \sum_{i=1}^n \sqrt{x_i^2 + \eps^2}.
\]
Since $\eps>0$, the function $J(x, \eps)$ is strictly convex in $x$. 
Hence, the minimizer in $x$ over any convex set is unique if it exists. 
For each $\eps\ge 0$, set
\[
x^\epsilon = \argmin_{x \in \Phi^{-1}(y)} J(x, \eps).
\]
The smoothing function $J(x, \eps)$ is used to measure the progress of the iteratively re-weighted 
iterates. For this we require that $\Phi$ satisfies the null space property NSP.
%
\begin{assu}\cite[Section 3]{cohen2009compressed}\label{as:nsp}
	\textbf{Null Space Property (NSP)} A matrix $\Phi \in \mathbb{R}^{\ell \times N}$ satisfies NSP of order $K$ for $\gamma \in (0,1)$ if and only if 
	\begin{equation}
		\norm{z_T}_1 \leq \gamma \norm{z_{T^c}}_1 \quad\forall\, z \in \text{Null}(\Phi)\label{3}
	\end{equation}
	and for all index sets $T\subset\{1,\dots,N\}$ of cardinality not exceeding $K$.
\end{assu}
\smallskip

\noindent
Observe that since \eqref{3} holds 
for all index sets $T\subset\{1,\dots,N\}$ of cardinality $K$, we must
have $K< N/2$. 
The null space property is intimately connected to the $k$-sparsity of solutions to the basis pursuit problem \ref{1}.
\begin{lemma}[NSP $+\ K$-sparsity imply uniqueness]
\cite[Lemma 4.3]{daubechies2010iteratively}\label{lm_cs_1}
Assume A2 holds and $\Phi^{-1}(y)$ contains an $K$-sparse vector $x^*$. 
Then $x^*$ is the unique $\ell_1$-minimizer in $\Phi^{-1}(y)$ and for all $v \in \Phi^{-1}(y)$, 
	\[
		\norm{v- x^*}_1 \leq 2\frac{1+\gamma}{1-\gamma} \sigma_L(v).
	\]
\end{lemma}



We now show that the null space property guarantees the boundedness of any sequence generated by Algorithm \ref{algorithm_cs}.

\begin{lemma}[Boundness of $\{x^n\}$]\label{lem:bound}
Let Assumption \ref{as:nsp} hold, and suppose $\{x^n\}$ is a sequence generated by Algorithm
\ref{algorithm_cs}. 
Then the sequence $\{J(x^n, \eps_n)\}$ is non-increasing, $\norm{x^n}_1 \leq J(x^0, \eps_0)$,
for all $n\in\bN$, and $\sum_{i=1}^\infty \norm{x^{n+1} - x^n}^2_{w^n} < \infty$.
\end{lemma}

\begin{proof}
By concavity of the square root function $\sqrt{b}+\frac{1}{2\sqrt{b}}(a-b)\ge\sqrt{a}$ for $0\le a,b$, and so
\begin{equation}\label{11}
	J(x^{n+1}, \eps_n) - J(x^n, \eps_n) \leq \frac{1}{2}(\norm{x^{n+1}}_{w^n}^2 - \norm{x^n}_{w^n}^2).
\end{equation}
By completing the square and rearranging terms, we have
\begin{equation}\label{12}
\begin{aligned}
\norm{x^{n+1}}_{w^n}^2 - \norm{x^n}_{w^n}^2& = -\norm{x^{n+1} - x^n}_{w^n}^2 \\
& + 2\ip{x^{n+1}}{x^{n+1} - x^n}_{w^n}.
\end{aligned}
\end{equation}
Since $x^{n+1} = \argmin_{x \in \Phi^{-1}(y)} \norm{x}_{w^n}$, we know
\begin{equation}\label{13}
\ip{x^{n+1}}{x^{n+1} - x^n}_{w^n} = 0.
\end{equation}
By combining \eqref{11}, \eqref{12} and \eqref{13} and using the fact that $\{\eps_n\}$ is non-increasing, we have
\[
\begin{aligned}
	J(x^{n+1},\eps_{n+1}) - J(x^n, \eps_n) &\leq J(x^{n+1}, \eps_n) -  J(x^{n}, \eps_n) \\
	& \leq -\frac{1}{2}\norm{x^{n+1} - x^n}_{w^n}^2.
\end{aligned}
\]
Hence $\norm{x^n}_1 \leq J(x^{n}, \eps_n) \leq J(x^0, \eps_0)$. Moreover, by telescoping we know 
\[\sum_{n=1}^\infty \norm{x^{n+1} - x^n}_{w^n}^2 \leq 2J(x^0, \eps_0) < \infty.\]
\end{proof}

\noindent
Our convergence proof also relies on the following lemma.
\begin{lemma}\cite[Lemma 4.2]{daubechies2010iteratively}\label{lm_4}
Let Assumption \ref{as:nsp} hold. Then, for any $z, z^\prime\in\Phi^{-1}(y)$, we have
\begin{equation}\label{eq:lm_4}
\norm{z - z^\prime}_1 \leq \frac{1-\gamma}{1+\gamma} 
\left[\,\norm{z^\prime}_1 - \norm{z}_1 + 2\sigma_K(z)\right],
\end{equation}	
where $\sigma_K$ is defined in \eqref{eq:sig}.	
\end{lemma}

The main convergence result   
makes use of the following notation:
for $S \subseteq [N]:= \{1,2,3,...,N\}$ and $x \in \mathbb{R}^N$, 
define $x_S \in \mathbb{R}^N$ componentwise by 
\[
(x_S)_i = 
\begin{cases}
x_i,&i\in S,\\
0,&\text{otherwise.}
\end{cases}
\]

\begin{theorem}[Convergence of Algorithm \ref{algorithm_cs}]\label{th_cs_1}
Let Assumption \ref{as:nsp} hold,
and let $y\in \mathbb{R}^m$ and $x_0\in\R^N$ be given. If
$\{x_k\}$ is generated by Algorithm \ref{algorithm_cs} initialized at $x_0$,
then there is an $\bx\in\R^N$ such that $x_k\rightarrow \bx$.
Moreover, the following hold.\\
	(1) If $\epsilon := \lim_{n \rightarrow \infty} \epsilon_n = 0$, then $\bar{x}$ is 
	$K$-sparse in which case $\bar{x}$ is the unique $\ell_1$ - minimizer.\\
	(2) If there exists a $K$-sparse $x^* \in \Phi^{-1}(y)$ , then 
	$\bar{x} = x^*$ is the unique $\ell_1$ - minimizer and $\lim_{n \rightarrow \infty} \epsilon_n = 0$. 
\end{theorem}
\begin{proof}
Part (1): The proof the part (1) is similar to the proof of \cite[Theorem 5.3(i)]{daubechies2010iteratively}. 
	First observe that $\epsilon$ is well-defined since the sequence $\{\eps_n\}_{n=1}^\infty$ is non-increasing. Moreover, by definition, $\sigma_K(x) = 0$ if and only if $x$ is $K$-sparse. Consequently if for any iteration $n_0$ we have $\epsilon_{n_0 + 1} = 0$, then Algorithm \ref{algorithm_cs} terminates at $x^{n_0}$ with $x^{n_0}$ $K$-sparse, and so 
 part (1) follows from Lemma \ref{lm_cs_1}. Therefore, we assume
 that the algorithm does not terminate and $0< \eps_n \rightarrow 0$. In this case, there 
 must be a subsequence $\cN\subset\bN$ 
 such that $\sigma_K(x^{n}) \overset{\cN}{\rightarrow} 0$. 
 Since Lemma \ref{lem:bound} tells us that the sequnce $\{x^n\}$ is bounded,
 there is a further subsequence $\cN'\subset \cN$ and a point $\bx\in\Phi^{-1}(y)$ such that
 $x^n\overset{\cN'}{\rightarrow}{\bx}$ with $\sigma_K(\bx)=0$. Hence, by Lemma \ref{lm_cs_1},
 $\bx$ is the unique $K$-sparse $\ell_1$-minimizer.
 
 Next let $\cJ\subset\bN$ be any subsequence. Again, by 
 Lemma \ref{lem:bound}, there is a further subsequence $\cJ'\subset\cJ$ and a point
 $x'$ such that $x^n\overset{\cJ'}{\rightarrow}x'$. Let $i\in\cN'$ and $j\in\cJ'$ be such that
 $i<j$. Then
\begin{equation*}
\begin{aligned}
\norm{x^{i} - x^{j}}_1
& \leq \frac{1-\gamma}{1+\gamma}(\norm{x^{j}}_1 - \norm{x^{i}}_1 + 2\sigma_K(x^{i})) 
\quad\mbox{(by \eqref{eq:lm_4})}
\\
& \leq\frac{1-\gamma}{1+\gamma}
          (J(x^{j}, \eps_{j}) - J(x^{i}, \eps_{i})+ N\eps_{i} + 2\sigma_K(x^{i})) 
\\
& \leq  \frac{1-\gamma}{1+\gamma}(N\eps_{i} + 2\sigma_K(x^{i})).
\quad\mbox{(by Lemma \ref{lem:bound})}
	\end{aligned}	
	\end{equation*}
Consequently, $\bx=x'$. Hence the entire sequence $\{x^n\}$ must converge to $\bx$ since 
every subsequence has a further subsequence convergent to $\bx$.
	
\noindent
Part (2):  First we assume $\eps = \inf_n \eps_n=\lim_{n \rightarrow \infty} \eps_n > 0$ and
establish a contradiction.
By Lemma \ref{lem:bound}, every
subsequence $\cN\subset\bN$ has a further subsequence $\cN'\subset\cN$ such
that $x^n\overset{\cN'}{\rightarrow}\tx$ for some $\tx\in\Phi^{-1}(y)$. 
For any $x\in \Phi^{-1}(y)$ and $i\in\cN'$, we have 
\begin{align}
J(x, \eps_{i}) - & J(x^{i}, \eps_{i})  \geq \ip{x^{i}}{x - x^{i}}_{w^{i}} \label{d1}
\\
& = \ip{x^{i+1}}{x - x^{i}}_{w^{i}}+\ip{x^{i}-x^{i + 1}}{x - x^{i}}_{w^{i}}  \nonumber \\
& \geq \ip{x^{i+1}}{x - x^{i}}_{w^{i}}-\norm{x^{i}-x^{i + 1}}_{w^{i}}\norm{x - x^{i}}_{w^{i}},\label{d2}
\end{align}
where \eqref{d1} follows from the convexity of $\sqrt{(\cdot)^2 + \epsilon_{i}^2}$ 
and \eqref{d2} is the Cauchy-Schwartz inequality.
Since $x^{i + 1} = \argmin_{x \in \Phi^{-1}(y)} \norm{x}_{w^{i}}^2$, we have
$\ip{x^{i+1}}{x - x^{i}}_{w^{i}}=0$. 
In addition, since $\eps=\inf_n \eps_n$, we have
$\norm{x - x^{i}}_{w^{i}}\le \eps^{-1}\norm{x - x^{i}}$. By combining these two statements
with \eqref{d2}, we obtain
\[
J(x, \eps_{i}) - J(x^{i}, \eps_{i})\ge -\eps^{-1}\norm{x^{i}-x^{i + 1}}_{w^{i}}\norm{x - x^{i}}.
\]
Since, by Lemma \ref{lem:bound}, $\norm{x^{i}-x^{i + 1}}_{w^{i}}\rightarrow 0$, we find that
\( J(x, \eps)\ge J(\tx, \eps)\). Consequently, $\tx=x^\eps$, that is, every subsequence
of $\{x^n\}$ has a further subsequence convergent to $x^\eps$ which implies that
the entire sequence converges to $x^\eps$.

Now set $T := \{i | x^*_i \neq 0, 1 \leq i \leq N \}$ so that $|T|\le K$,
and observe that
\begin{equation}\label{cs_1}
\norm{x^\epsilon}_1 \leq J(x^\eps, \eps) \leq J(x^*, \eps) \leq \norm{x^*}_1 + N\epsilon .
\end{equation}
In addition, we have
\begin{align}\nonumber 
\norm{x_{T^c}^\epsilon}_1 
& = \norm{x^\epsilon}_1 - \norm{x^\epsilon_{T}}_1 
\\
&\leq \norm{x^*}_1 + N\eps -(\norm{x^*_T}_1 - \norm{x^*_T - x_T^\epsilon}_1) \\ 
& \qquad \text{(by \eqref{cs_1} and $\Delta$ inequality)} \nonumber 
\\
&\le N\eps+\norm{x^*_T - x_T^\epsilon}_1 \\ 
& \qquad(\text{since }\norm{x^*}_1=\norm{x^*_T}_1) \nonumber
\\
& \leq \gamma \norm{x_{T^c}^\epsilon} + N\epsilon \ .\\
& \qquad \mbox{(NSP)} \label{cs_2}
\end{align}
Next observe that 
\[
\begin{aligned}
N\eps = \lim_{n \rightarrow \infty} N\epsilon_n 
&\leq \lim_{n \rightarrow \infty} \eta(1-\gamma)\sigma_K(x^n)  \\
&= \eta(1-\gamma)\sigma_K(x^\eps) \leq \eta(1-\gamma)\norm{x^\epsilon_{T^c}}_1 .
\end{aligned}
\] 
Plugging this into \eqref{cs_2} gives
\begin{equation}\label{eq:ineq for x eps}
\norm{x_{T^c}^\epsilon}_1\le \gamma \norm{x_{T^c}^\epsilon} +\eta(1-\gamma)\norm{x^\epsilon_{T^c}}_1.
\end{equation}
If $\norm{x^\epsilon_{T^c}}_1=0$, then $x^\epsilon=x^*$ and $\sig_K(x^\epsilon)=0$.
But then $\lim_n\sig_K(x^n)=\sig_K(x^\epsilon)=0$ which implies that 
$\eps_n \rightarrow 0$, a contradiction. Therefore, $\norm{x^\epsilon_{T^c}}_1>0$.
Dividing \eqref{eq:ineq for x eps} by $\norm{x^\epsilon_{T^c}}_1$ gives
	\[
		1\le \gam +\eta(1-\gam)<\gam +(1-\gam)=1
		\qquad\mbox{(since $\eta\in (0,1)$)}
	\]
	a contradiction. 
Therefore, $\eps$ must equal zero which returns us to Part (1) and completes the proof.
\end{proof}

\noindent
We now establish the local linear convergence for Algorithm \ref{algorithm_cs}.
Recall that a sequence $\{z^k\}\subset\R^N$ converges 
\emph{locally linearly}
to $z^*\in \R^N$ if there are constants $\kappa\ge 0$ and $\lam\in(0,1)$ and an iteration $k_0\in\bN$
such that 
\[
\norm{z^k-z^*}\le \kappa\lam^{k-k_0}\norm{z^{k_0}-z^*}\quad\forall\, k\ge k_0.
\]
In \cite{daubechies2010iteratively}, the authors refer to linear convergence as 
\emph{exponential convergence}.

\begin{theorem}[The Local Linear Convergence of Algorithm \ref{algorithm_cs}]
\label{th_cs_2}
Let Assumption \ref{as:nsp} hold, and suppose that 
$\Phi^{-1}(y)$ contains a $K$-sparse vector $x^*$. 
Set $T := \{i | x^*_i \neq 0, 1 \leq i \leq N \}$ and choose
$\rho\in \left(0,\, 1-\gamma(1+\eta(1-\gamma))\right)$, where $\gamma$ is given in 
A2 and $\eta\in(0,1)$
is initialized in Algorithm \ref{algorithm_cs}.
Then there is a smallest $n_0\in\bN$ such that 
	\begin{equation}\label{eq:n zero}
		\norm{(x^{n_0} - x^*)_{T^c}}_1 \leq \rho \min_{i \in T}|x^*_i|\ .
	\end{equation}
Moreover, for all $n \geq n_0$, 
	\begin{eqnarray}
		&\norm{(x^{n+1} - x^*)_{T^c}}_1 \leq \mu \norm{(x^n-x^*)_{T^c}}_1,
		\label{eq:gc1}\\
		&\norm{x^n - x^*}_1 \leq (1+\gamma)\mu^{n-n_0} \norm{x^{n_0} - x^*}_1,
		\label{eq:gc1b}
	\end{eqnarray}
	where $\mu := \frac{\gamma(1+\eta(1-\gamma))}{1-\rho} < 1$.
\end{theorem}

\begin{proof}
By Theorem \ref{th_cs_1}, $x^n\rightarrow x^*$ so that for every 
$\rho\in \left(0,\, 1-\gamma(1+\eta(1-\gamma))\right)$ there is a smallest $n_0\in\bN$ such that
\eqref{eq:n zero} holds. Consequently, $n_0$ exists.

We follow the proof in \cite[Theorem 6.1]{daubechies2010iteratively}. 
We prove \eqref{eq:gc1} by induction. Let $\hn \geq n_0$ be such that \eqref{eq:n zero} holds
with $n_0$ replaced by $\hn$.
Since $x^{\hn+1} = \argmin_{x \in \Phi^{-1}(y)} \norm{x}_{w^\hn}^2$, 
the optimality conditions for this problem tell us that
	\[
		\ip{x^{\hn+1}}{x^{\hn+1} - x^*}_{w^\hn} = 0.
	\]
Consequently,
 		\[
		\begin{aligned}
 			\norm{x^{\hn+1}\! -\! x^*}_{w_\hn}^2 & =-\ip{x^*}{x^{\hn+1} \!-\! x^*}_{w_\hn}  \\ 
			&= - \ip{(x^*)_T}{x^{\hn+1}\! -\! x^*}_{w_\hn}\ \\
			& \leq\! \sum_{i \in T} \frac{|x_i^*(x^{\hn+1}_i - x^*_i)|}{\sqrt{(x^\hn_i)^2 + \epsilon_\hn^2}}.
			\end{aligned}
 		\]
 Note, for $i \in T$, NSP tells us that
 \[
 |x^\hn_i - x^*_i| \leq \norm{(x^\hn - x^*)_T}_1 \leq \gamma \norm{(x^\hn - x^*)_{T^c}} 
 \leq \rho \min_{i \in T}|x^*_i|,
 \] 
 we have
 		\[
 			\frac{|x^*_i|}{\sqrt{(x^\hn_i)^2 + \eps_\hn^2}} \leq \frac{|x^*_i|}{|x^\hn_i|} \leq \frac{|x^*_i|}{|x^*_i| - |x^\hn_i - x^*_i|} \leq \frac{1}{1-\rho}.
 		\]
 		Hence
 		\[
		\begin{aligned}
 		 \norm{x^{\hn+1} - x^*}_{w_\hn}^2 &\leq \frac{1}{1-\rho} \norm{(x^{\hn+1} - x^*)_T}_1 \\
		 & \leq \frac{\gamma}{1-\rho} \norm{(x^{\hn+1} - x^*)_{T^c}}_1.
		 \end{aligned}
 		\]
 		Consequently, by Cauchy-Schwartz Inequality,
		\[
 		\begin{aligned}
 	         \|(x^{\hn+1} - &x^*)_{T^c}\|_1^2 
		=\left(\sum_{i \in T^c}\frac{|x^{\hn+1}_i - x^*_i|}{((x_i^\hn)^2 + \eps_\hn^2)^{1/4}}
		((x_i^\hn)^2 + \eps_\hn^2)^{1/4}\right)^2
		\\
		&= \norm{(x^{\hn+1} - x^*)_{T^c}}_{w_\hn}^2\left(\sum_{i \in {T^c}} 
		\sqrt{(x_i^\hn)^2 + \eps_\hn^2}\right) 
		\\
		&\leq \norm{x^{\hn+1} - x^*}_{w_\hn}^2\left(\sum_{i \in {T^c}}|x^\hn_i|+\eps_\hn\right)
		\\
 		& \leq \frac{\gamma\norm{(x^{\hn+1} - x^*)_{T^c}}_1}{1-\rho}  \left[\norm{(x^\hn - x^*)_{T^c}}_1 + N\eps_\hn\right].
 		\end{aligned}
		\]
 		Therefore
 		\begin{align*}
 			\norm{(x^{\hn+1} - x^*)_{T^c}}_1 & \leq \frac{\gamma}{1-\rho} [\norm{(x^\hn - x^*)_{T^c}}_1 + N\eps_\hn] \\
 			& \leq \frac{\gamma}{1-\rho} [\norm{(x^\hn - x^*)_{T^c}}_1 + \eta(1-\gamma) \sigma_K(x^\hn)].\\
			&\qquad \mbox{(Step 4 in Algorithm \ref{algorithm_cs})}
  		\end{align*}
 Observe $\sigma_K(x^\hn) \leq \norm{(x^{\hn})_{T^c}}_1 = \norm{(x^\hn - x^*)_{T^c}}_1$. 
 Hence
\begin{equation}
\label{eq:gc1 goal}
\begin{aligned}
 			\norm{(x^{\hn+1} - x^*)_{T^c}}_1 & \leq \frac{\gamma(1+\eta(1-\gamma))}{1-\rho} \norm{(x^\hn - x^*)_{T^c}}_1\\
			&=\mu\norm{(x^\hn - x^*)_{T^c}}_1.
	\end{aligned}
\end{equation}
Since $n_0$ satisfies \eqref{eq:n zero}, this shows that \eqref{eq:gc1} is satisfied for
$\hn=n_0$.

Now assume \eqref{eq:gc1} holds for $\{n_0, n_0+1, ..., n-1\}$. Then \eqref{eq:gc1} tell us that 
\begin{equation}\label{eq:gc2}
\begin{aligned}
	\norm{(x^{n} - x^*)_{T^c}}_1 & \leq \mu \norm{(x^{n-1} - x^*)_{T^c}}_1 \\
	& \leq ...  \leq \mu^{n-n_0} \norm{(x^{n_0} - x^*)_{T^c}}_1 \\
	&  \leq \rho \min_{i \in T}|x_i^*|,
\end{aligned}
\end{equation} 
where the last inequality is by \eqref{eq:n zero} and $\mu < 1$. 
In particular, we have \eqref{eq:n zero} with $n_0$ replaced by $n$, and so, by 
\eqref{eq:gc1 goal},
\eqref{eq:gc1} is satisfied at $n$ which completes the induction.

 Finally, the NSP for $\Phi$ tells us that
\[
\begin{aligned}
 \norm{x^n - x^*}_1 & \leq (1+\gamma)\norm{(x^n - x^*)_{T^c}}_1 \\ 
 & \leq (1+\gamma)\mu^{n-n_0}\norm{(x^{n_0}-x^*)_{T^c}}_1  \\
 &\leq (1+\gamma)\mu^{n-n_0}\norm{x^{n_0}-x^*}_1.
\end{aligned}
\]
\end{proof}

\section{Failure of DDFG-IRLS}
We construct an example where 
the DDFG-IRLS algorithm provably fails for $K-2\gam/(1-\gam)\le \gam \le K$.
However, we emphasize that, in general,  
the failure of this inequality does not imply the failure of the DDFG-IRLS algorithm.
%

The example is formulated in the context of the $\ell_1$ regression problem $\ell_1$R discussed
in the introduction.
It is well-known that \ref{1} is equivalent to this $\ell_1$ regression problem 
under the corresponces $\text{rge}\,(A) = \text{Null}(\Phi)$ and $\Phi b = -y$ \cite{candes2005decoding}. 
In addition, under these correspondences, 
the NSP for $\Phi$ of order $K$ for 
$\gamma \in (0, 1)$ is equivalent to the following condition on the matrix $A$:
\begin{equation}\label{nsp_l1} 
\norm{(Az)_T}_1 \leq \gamma \norm{(Az)_{T^c}}_1\quad
\text{for all $z$ and all $|T| \leq K$.}
\end{equation}

In terms of the DDFG-IRLS algorithm, when the matrix $A$ has full column rank,
then there is a 1-1 correspondence between the 
iterates of this algorithm and a corresponding IRLS algorithm for solving the \ref{pr:l1}.
If we denote the $i$th row of $A$ by $a_i$, for given $\eps_0$ and $x^0$, 
this correspondence is given by 
\[
x^{n}  = A z^{n} - b \quad \forall\, n=0,1,\dots,
\] 
where, for $n=0,1,\dots$,
\begin{equation}\label{algorithm_example}
\begin{aligned}
\text{DDFG-IRLS}&\left\{
\begin{aligned}
x^{n+1} &:= \min_{x \in \Phi^{-1}(y)} \sum_{i=1}^N \frac{x_i^2}{\sqrt{(x^n_i)^2 + \eps_n^2}} \\\
\epsilon_{n+1} &:= \min\left\{\eps_n, \frac{r_{K+1}(x^{n+1})}{N}\right\}
\end{aligned}
\right.
\\
\mbox{$\ell_1$R-IRLS}
&\left\{
\begin{aligned}
z^{n+1} & := \min_z \sum_{i=1}^N \frac{(a_i^Tz - b_i)^2}{\sqrt{(a_i^Tz^n - b_i)^2 + \epsilon_n^2}} \\
 \epsilon_{n+1} & := \min \left\{ \eps_n, \frac{r_{K+1}(A z^{n+1} - b)}{N}\right\}\, .
\end{aligned}	
\right.
\end{aligned}
\end{equation}
Therefore, by Lemma \ref{lm_cs_1}, 
whenever $\Phi$ satisfies the NSP of order $K$ for $\gam$, 
or equivalently, $A$ satisfies \eqref{nsp_l1}, if  there exists $z^*$ for which
$Az^*-b$ is $K$-sparse, then  $x^*:=Az^*-b$
is the unique solution to \ref{1}. If, in addition, $A$ has full column rank,
then $z^*$ is the unique solution to \ref{pr:l1}.

We now construct our example.
Given $k \geq 1$, 
set $\tilde{A} := (I_k, ..., I_k)^T \in \mathbb{R}^{(2k^2+k) \times k}$ with $2k+1$ blocks of 
the identity $k\times k$ matrix $I_k$. For any $z \in \mathbb{R}^k$ and any $T \subseteq [2k^2 + k]$ with $|T| = k$, let $i_0 \in \set{i }{ |z_i| \geq |z_j| \ \forall\, 1 \leq j \leq k }$. 
Then  
\[
\norm{(\tilde{A}z)_T}_1 \leq k|z_{i_0}| = \frac{k}{k+1} (k+1)|z_{i_0}| \leq \frac{k}{k+1} \norm{(\tilde{A}z)_{T^c}}_1.
\] 
Thus, for $K=k$, $\tilde A$ satisfies \eqref{nsp_l1} with $\gamma = \frac{k}{k+1}$, and
this value for $\gam$ is sharp. 
We now modify $\tilde{A}$ to obtain an $A_\gam$ whose $\gam$ is any
element of $(\frac{k}{k+1},\ 1)$.
To this end, let $\gam\in (\frac{k}{k+1},\ 1)$ and define
$A_\gamma \in \mathbb{R}^{(2k^2 + k) \times k}$ 
so that $A_\gamma(ik+1, 1) := \frac{k+1}{k} \gamma$ for all $0 \leq i \leq k-1$, while 
all other components of $A_\gam$ coincide with those of $\tilde{A}$.
That is, we only replace the $(1,1)$ entry in each of the first $k$ identity matrices $I_k$ 
of $\tilde{A}$ by $\frac{k+1}{k} \gamma$.
By applying the same argument to $A_\gam$ as above for $\tilde A$, we find that
$A_\gam$ satisfies \eqref{nsp_l1} with $K=k$ and $\gamma = \frac{k}{k+1}$,
and this $\gam$ is also sharp.

Next choose $z^* \in \R^k_{++}$. 
Given $\del\in\R$, set $b:= A_\gamma z^* + \delta \te$, where
$\te:=\sum_{j=0}^{k-1}e_{(jk+1)}$ with each $e_{(jk+1)}$ the ${(jk+1)}$th standard unit coordinate vector.
Observe that $x^*:=A_\gam z^*-b$ is $k$-sparse and $A_\gam$ has full column rank.
Hence, by our previous discussion, Lemma \ref{lm_cs_1} implies that
$z^*$ is the unique solution to \ref{pr:l1} for this choice of $A$ and $b$.
Our goal is to show that there is an initialization for the $\ell_1$R-IRLS algorithm in \eqref{algorithm_example}
such that the generated sequence $\{z^n\}$  
satisfies $z^n \nrightarrow z^*$, and 
hence, the corresponding DDFG-IRLS iterates 
$x^n := A_\gamma z^n - b$ do not converge to the unique solution $x^*:= A_\gamma z^* - b $ to
\ref{1}.

\begin{theorem}\label{thm:ex2}
Let $z^*\in\R^k$, $\delta \in \left(0, k(2k+1)\right]$, and $\gamma \in \left[\nu,  1 \right)$, where
\[
\nu:=\sqrt{\frac{1+\frac{1}{4k^2(2k+1)^2}}{1+\frac{1}{k^2(2k+1)^2}}}
=
\sqrt{\frac{4k^2(2k+1)^2 + 1}{4k^2(2k+1)^2 + 4}}.
\]
For these values of $z^*$, $\gam$ and $\del$,
let $A_\gam$ and $b$ be as given above and
consider the problem \ref{pr:l1} having unique solution $z^*$. 
Define 
\[
\alpha := \gamma \frac{k+1}{k}\quad\text{ and }\quad
\xi := \gamma \sqrt{1+\frac{1}{k^2(2k+1)^2}}.
\]
Then
\begin{equation}\label{eq:ex2 ineq}
\begin{aligned}
\alpha >1\, & ,\quad\xi\geq\sqrt{1+(4k^2(2k+1)^2)^{-1}}\ >1\quad\text{and}\\ 
& \gamma/(k(2k+1)\sqrt{\xi^2 - 1})>1.
\end{aligned}
\end{equation}
Initialize $\eps_0 := 1$ and $z^0\in\R^k_{++}$ componentwise by
\begin{equation}\label{eq:x0}
\begin{aligned}
z^0_1  &\in \left(z^*_1 +  \frac{\delta}{\alpha + \gamma/(k(2k+1)\sqrt{\xi^2 - 1})},\ z^*_1 
+  \frac{\delta}{\alpha + 1} \right)\quad \text{and}
\\
z^0_i &:= z^*_i,\, i =2,\dots, k\, . 
\end{aligned}
\end{equation}
If $\{z^n\}$ is the sequence generated by the $\ell_1$R-IRLS algorithm in
\eqref{algorithm_example} with this initialization, then
$z^n \nrightarrow z^*$.
\end{theorem}
\begin{proof}
We first prove the inequalities in \eqref{eq:ex2 ineq}.
The first inequality follows since 
	\begin{equation*}
		\begin{aligned}
			\alpha > 1 & \Longleftarrow \nu^2 < \left(\frac{k}{k+1}\right)^2 \\
			& \Longleftrightarrow (k+1)^2\left(1 + \frac{1}{4k^2(2k+1)^2} \right) > k^2\left(1 + \frac{1}{k^2(2k+1)^2} \right) \\
			& \Longleftrightarrow 2k+1+\frac{(k+1)^2}{4k^2(2k+1)^2} > \frac{1}{(2k+1)^2}.
		\end{aligned}
	\end{equation*}
The second inequality in \eqref{eq:ex2 ineq}
follows directly from the fact that $\gamma \geq \nu$.
The third inequality in \eqref{eq:ex2 ineq} follows since
	\begin{equation*}
		\begin{aligned}
			\gamma/\left(k(2k+1)\sqrt{\xi^2 - 1}\right)>1 & \Longleftrightarrow \xi^2 < 1 + \frac{\gamma^2}{k^2(2k+1)^2}\\
			& \Longleftrightarrow \gamma^2 < 1.
		\end{aligned}
	\end{equation*}

Note that the third inequality in \eqref{eq:ex2 ineq} implies that
\[
\delta[\alpha + \gamma/(k(2k+1)\sqrt{\xi^2 - 1})]^{-1}< \delta(\alpha + 1)^{-1}
\]
so that $x_1^0$ is well defined.

We establish the result by showing that
$z^n_1 \nrightarrow z^*_1$. 
 Observe that 
 \begin{equation}\label{eq:bees}
 \begin{aligned}
b_{k^2+1} &= (A_\gamma z^*)_{k^2 + 1} = z^*_1  \text{ and }\   \\
b_1& =\alpha z^*_1 + \delta=\alf b_{k^2 + 1} + \del.
\end{aligned}
 \end{equation}
 By the $\ell_1$R-IRLS algorithm, $z^{n+1}$ solves the 
 least-squares problem
\[
\begin{aligned}
\min_z \  
\frac{k(\alpha z_1 - b_1)^2}{\sqrt{(b_1 - \alpha z^n_1)^2 + \epsilon_n^2}} 
&+ 
\frac{(k+1)(z_1 - b_{k^2+1})^2}{\sqrt{(z^n_i - b_{k^2+1})^2 + \epsilon_n^2}}\\
&+
(2k+1)\sum_{i=2}^k \frac{(z_i - b_i)^2}{\sqrt{(z^n_i - b_i)^2 + \epsilon_n^2}} .
\end{aligned}
\]
Due to the separability of the objective in the variables $z_i,\ i=2,\dots, k$, we have
$z^n_i = b_i = z^*_i,\ i = 2,\dots,k$, for $n \geq 1$. 
The optimality conditions for each subproblem tells us that
\begin{equation}\label{counter_example_1}
z^{n+1}_1 = \frac{\frac{\alpha b_1 k}{\sqrt{(b_1 - \alpha z^n_1)^2 +\epsilon_n^2}} + \frac{(k+1) b_{k^2+1} }{\sqrt{(z^n_1 - b_{k^2+1})^2 + \epsilon_n^2}}}{ \frac{k\alpha^2}{\sqrt{(b_1 - \alpha z^n_1)^2 +\epsilon_n^2}} + \frac{(k+1)}{\sqrt{(z^n_1 - b_{k^2+1})^2 + \epsilon_n^2}}}\ .
\end{equation}
By \eqref{counter_example_1}, we have
\begin{equation}
	\begin{aligned}\label{eq:appendix3}
		z^{n+1}_1 - b_{k^2+1} & = \frac{\frac{\alpha k (b_1 - \alpha b_{k^2+1})}{\sqrt{(b_1 - \alpha z^n_1)^2 + \eps_n^2}}}{\frac{k\alpha^2}{\sqrt{(b_1 - \alpha z^n_1)^2 +\epsilon_n^2}} + \frac{(k+1)}{\sqrt{(z^n_1 - b_{k^2+1})^2 + \epsilon_n^2}}} \\
		& = \frac{\frac{\alpha k \delta}{\sqrt{(b_1 - \alpha z^n_1)^2 + \eps_n^2}}}{\frac{k\alpha^2}{\sqrt{(b_1 - \alpha z^n_1)^2 +\epsilon_n^2}} + \frac{(k+1)}{\sqrt{(z^n_1 - b_{k^2+1})^2 + \epsilon_n^2}}}  \\
		& \geq 0
	\end{aligned}
\end{equation}
and
\begin{equation}
	\begin{aligned}\label{eq:appendix4}
		b_1 - \alpha z^{n+1}_1 & = \frac{\frac{(k+1)(b_1 - \alpha b_{k^2+1})}{\sqrt{(z^n_1 - b_{k^2+1})^2 + \eps_n^2}}}{\frac{k\alpha^2}{\sqrt{(b_1 - \alpha z^n_1)^2 +\epsilon_n^2}} + \frac{(k+1)}{\sqrt{(z^n_1 - b_{k^2+1})^2 + \epsilon_n^2}}} \\
		& = \frac{\frac{(k+1)\delta}{\sqrt{(z^n_1 - b_{k^2+1})^2 + \eps_n^2}}}{\frac{k\alpha^2}{\sqrt{(b_1 - \alpha z^n_1)^2 +\epsilon_n^2}} + \frac{(k+1)}{\sqrt{(z^n_1 - b_{k^2+1})^2 + \epsilon_n^2}}}  \\
		&\geq 0.
	\end{aligned}
\end{equation}
Hence,
\begin{equation}\label{counter_example_2} 
\begin{aligned}
z^{n+1}_1 - b_{k^2+1} &\geq 0,\, b_1 - \alpha z^{n+1}_1 \geq 0,
\text{ and } \\
s_{n+1}
& = \gamma \sqrt{\frac{(z^n_1 - b_{k^2+1})^2 + \epsilon_n^2}{(b_1 - \alpha z^n_1)^2 + \epsilon_n^2}},
\quad \forall\,  n \geq 0,
\end{aligned}
\end{equation}
where $s_{n+1} :=   (z^{n+1}_1 - b_{k^2+1})/(b_1 - \alpha z^{n+1}_1)$.
	
If we let $\varepsilon_n := z^n_1 - b_{k^2+1}$, then 
$s_n= \varepsilon_n / (\delta - \alpha \varepsilon_n)$ by \eqref{eq:bees}. 
For $n=0$,  \eqref{eq:x0} tells us that
 \begin{equation}\label{counter_example_3}
\begin{aligned}
s_0 &= \frac{ \vareps_0}{\delta - \alpha \vareps_0} 
=\frac{1}{(\del/\vareps_0)-\alf} \\
& \in \left( \frac{k(2k+1)\sqrt{\xi^2 - 1}}{\gamma}, 1 \right).
\end{aligned}
\end{equation}

We now show by induction that 
\begin{equation}\label{eq:ex2 induction}
s_n > k(2k+1)\sqrt{\xi^2 - 1}\ \ \text{ and }\ \
\epsilon_n = \vareps_n/(k(2k+1))\qquad\forall \ n \geq 1\ .
\end{equation}
First consider $n = 1$.
Since $\epsilon_0 = 1$, the definition of $\vareps_0$ and $s_0$
in conjunction with \eqref{eq:bees} and  \eqref{counter_example_2} 
tell us that $s_1 = \gamma\sqrt{\frac{ \vareps_0^2 + 1}{(\delta - \alpha \vareps_0)^2 + 1}}$ and so,
by \eqref{counter_example_3}
\begin{equation}\label{counter_example_5}
	s_1 = \gamma\sqrt{\frac{ \vareps_0^2 + 1}{(\delta - \alpha \vareps_0)^2 + 1}} \leq \gamma < 1.
\end{equation}
Observe that
\[
(A_\gam z^n-b)_i=
\begin{cases}
\alf z^n_1-b_1,&\text{if } i\in\set{jk+1}{ j\in\{0,\dots,k-1\}},
\\
z^n_1-b_{k^2+1},&\text{if } i\in\set{jk+1}{ j\in\{k,\dots,2k\}},
\\
0,&\text{otherwise.}
\end{cases}
\]
Hence, since $(z^1_1 - b_{k^2+1})/(b_1 - \alpha z^1_1)=s_1 < 1$, 
the $(k+1)$th largest magnitude of the entries of $A_\gamma z^1 - b$ is $|z^1_1 - b_{k^2+1}|$
with $|z^1_1 - b_{k^2+1}|=z^1_1 - b_{k^2+1}$     by \eqref{counter_example_2}.
Thus $\epsilon_{1} = \min\left\{\epsilon_0, \frac{z^{1}_1 - b_{k^2+1}}{k(2k+1)}\right\}$. The given definitions and the inequality
$s_1<1$, yield
\[
z^1_1 - b_{k^2+1} = \varepsilon_1 = \frac{\delta s_1}{\alpha s_1 + 1} 
=\frac{\delta }{\alpha  + (1/s_1)}
\leq \frac{\delta}{\alpha + 1} \leq k(2k+1).
\]
Therefore, $\eps_1 = \frac{\varepsilon_1}{k(2k+1)}$, since 
$\eps_0=1$, which proves the second part of
\eqref{eq:ex2 induction} for $n=1$.  
To obtain the first part of \eqref{eq:ex2 induction} for $n=1$,
observe that
\begin{equation}\label{counter_example_6}
		\begin{aligned}
			\frac{s_1^2}{k^2(2k+1)^2} + 1 & = \frac{\gamma^2}{k^2(2k+1)^2} \frac{ \vareps_0^2 + 1}{(\delta - \alpha \vareps_0)^2 + 1} + 1 \\
			& \qquad \text{(by \eqref{counter_example_5})}
\\
			& \geq \frac{\gamma^2}{k^2(2k+1)^2} \frac{ \vareps_0^2}{(\delta - \alpha \vareps_0)^2} + 1 \\
			 & \qquad (\text{since $ \vareps_0 \leq \delta - \alpha \vareps_0$
			by \eqref{counter_example_3}}) 
\\
			& > \xi^2\, .  \\
			& \qquad  (\text{by lower bound in }\eqref{counter_example_3}) 
\\
		\end{aligned}
\end{equation}
	Thus, $s_1 > k(2k+1)\sqrt{\xi^2-1}$.
	
Assume $s_n > k(2k+1)\sqrt{\xi^2 -1}$ and $\epsilon_n = \frac{\varepsilon_n}{k(2k+1)}$. 
	Plugging $\epsilon_n = \frac{\varepsilon_n}{k(2k+1)}$ into \eqref{counter_example_2} gives
\begin{equation}\label{counter_example_7}
\begin{aligned}
		s_{n+1} &= \gam\sqrt{\frac{\vareps_n^2+\frac{\vareps_n^2}{k^2(2k+1)^2}}
		{(b_1-\alf z^n_1)^2+\frac{\vareps_n^2}{k^2(2k+1)^2}}}
		\\ & =
		\gam\sqrt{1+\frac{1}{k^2(2k+1)^2}}
		\sqrt{\frac{\vareps_n^2}
		{(b_1-\alf z^n_1)^2+\frac{\vareps_n^2}{k^2(2k+1)^2}}}
		\\ & = \xi \frac{s_n}{\sqrt{1 + \frac{s_n^2}{k^2(2k+1)^2}}}.
\end{aligned}
\end{equation}
Since the function $f(x) := \frac{x}{\sqrt{1 + x^2(k^2(2k+1)^2)^{-1}}}$ is increasing on $(0, \infty)$, 
we know
	\[
	s_{n+1} = \xi f\left(s_n\right) \geq \xi f\left(k(2k+1)\sqrt{\xi^2 - 1}\right) = k(2k+1)\sqrt{\xi^2 - 1},
	\]
which established the first part of \eqref{eq:ex2 induction} for $n+1$.
To establish the second part, observe that \eqref{counter_example_7} 
and the induction hypothesis gives
	\[
		s_{n+1} = \xi \frac{s_n}{\sqrt{1+s_n^2(k^2(2k+1)^2)^{-1}}} \leq \xi \frac{s_n}{\sqrt{1 + \xi^2 - 1}} = s_n ,
	\]
	
Thus far, we have shown that 
$s_{n+1} \geq k(2k+1)\sqrt{\xi^2 - 1}$ 
and  
$s_{n+1} \leq s_n$.
By combining these inequalities with the fact that $s_n = \frac{\varepsilon_n}{\delta_n - \alpha  \vareps_n}$ for each $n \geq 1$, 
we have $\varepsilon_{n+1} \leq  \vareps_n$ for all $n\ge 1$. 
Therefore, by the induction hypothesis,
 $\epsilon_{n+1} = \min\{\epsilon_n, \frac{\varepsilon_{n+1}}{k(2k+1)}\} = \min\{ \frac{\varepsilon_{n}}{k(2k+1)},  \frac{\varepsilon_{n+1}}{k(2k+1)}\} = \frac{\varepsilon_{n+1}}{k(2k+1)}$. 
 This concludes the proof of \eqref{eq:ex2 induction}.
	 
Observe that our induction proof also shows that $\{s_n\}$ is a non-increasing sequence bounded below
	 by $k(2k+1)\sqrt{\xi^2 - 1}$. 
	 Therefore, there is an $s^* \geq k(2k+1)\sqrt{\xi^2 - 1} > 0$
	 such that $s_n \downarrow s^*$. In particular, by taking the limit in \eqref{counter_example_7}, 
	 we have
	\(
	s^* = \xi s^*/\sqrt{1+\frac{(s^*)^2}{k^2(2k+1)^2}},
	\)
or equivalently, $s^* = k(2k+1)\sqrt{\xi^2 - 1}$. 
The induction showed that $s_n=\vareps_n(\del-\alf\vareps_n)$
and so $\vareps_n=(\del s_n)/(1+\alf s_n)$ which tells us that
\[
\begin{aligned}
z^n_1-z^*_1&=z^n_1-b_{k^2+1} =
\vareps_n
=(\del s_n)/(1+\alf s_n)
\\ &\rightarrow
(\del s^*)/(1+\alf s^*)
=
\frac{k(2k+1)\sqrt{\xi^2-1}}{1+\alf k(2k+1)\sqrt{\xi^2-1}}>0 .
\end{aligned}
\]
Consequently, $z^n_1\not\rightarrow z^*_1$,
and we have arrive at the desired result.
\end{proof}	

\section{Numerical Examples}\label{sec:exp}


\subsection{Failure of the DDFG-IRLS Algorithm} 
We present three numerical experiments illustrating the failure of the DDFG-IRLS algorithm 
for small perturbations of the example
given in Theorem  \ref{thm:ex2}.
Experiment 1 (see Figure \ref{cs_counter}) simply illustrates the content of Theorem \ref{thm:ex2} 
for $k = 5$, $\gamma = \sqrt{(4k^2(2k+1)^2 + 1)/(4k^2(2k+1)^2 + 4)} = 0.999876$, $\delta = k(2k+1) = 55$. 
The true solution of problem \ref{pr:l1}, $z^*$, is sampled from $N(0, I_k)$. 
In both algorithms, $x^0$ is initialized as $x^0 := A_\gamma z^0 - b$ where $z^0$ satisfies \eqref{eq:x0}, i.e., 
$z^0_1 = z^*_1 + (\delta/(\alpha + \gamma / (k(2k+1) \sqrt{\xi^2-1 })  ) + \delta/(\alpha + 1))/2$. 
For Algorithm \ref{algorithm_cs}, $\eta = 0.9$.
\begin{figure}[H]
  \centering
  \begin{minipage}[b]{0.49\textwidth}
    \includegraphics[width=\textwidth]{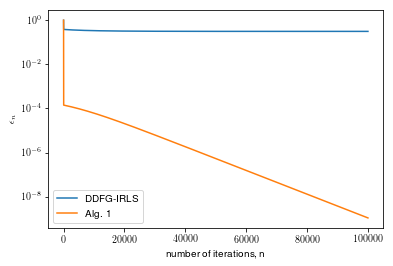}
  \end{minipage}
  \hfill
  \begin{minipage}[b]{0.50\textwidth}
    \includegraphics[width=\textwidth]{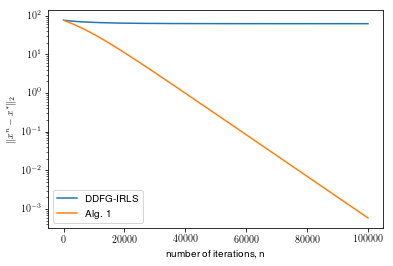}
  \end{minipage}
  \caption{Experiment 1: The performance of DDFG-IRLS versus Algorithm \ref{algorithm_cs} for the set-up in Theorem \ref{thm:ex2}. The left figure is $\eps_n$ versus the number of iterations $n$. The right figure is $\norm{x^n - x^*}_2$ versus the number of iterations $n$.}
 \label{cs_counter}
\end{figure}
In experiment 2 (see Figure \ref{cs_counter_2}), we examine the sensitivity of the success/failure of the DDFG-IRLS 
algorithm to the selection of the parameter $\gamma$ near the critical value 
$\gamma_0 := \sqrt{(4k^2(2k+1)^2 + 1)/(4k^2(2k+1)^2 + 4)} \approx 1-10^{-3.9}$. 
Again, we let $k = 5$. To illustrate the effect of the selection of $\gamma$, 
we run the DDFG-IRLS algorithm
for $\gamma \in \{1-10^{-1}, 1-10^{-2}, 1-10^{-3}, 1-10^{-3.3}, 1-10^{-3.6}, 1-10^{-\gamma_0}, 1-10^{-4}, 1-10^{-5}\}$. 
Here,  20 instances of the random variable $N(0, 100 \cdot I_5)$ are chosen for the starting point $z_0$. 
All other parameters are the same 
as those of experiment 1. The iterations are terminated when either $\norm{x^n - x^*} \leq 10^{-3}$ or the number of iterations exceeds $10^5$. 
In the range $1-10^{-3.6}\le \gam<\gam_0$, all the experiments fail to 
achieve the termination criteria $\norm{x^n - x^*}_2 \leq 10^{-3}$. 
This illustrates the extremely slow rate of
convergence of the DDFG-IRLS algorithm when the critical value $\gam_0$ is approached from below.

In experiment 3 (see Figure \ref{cs_counter_2}), we examine the robustness of the success/failure of the 
DDFG-IRLS algorithm for small perturbations of the example
given in Theorem  \ref{thm:ex2}
obtained by
perturbing the matrix $A_\gamma$. Again, we let $k = 5$ and $\delta = k(2k+1) = 55$ and use DDFG-IRLS to solve perturbed versions of our basic example with $A_{\gamma, \sigma} = A_\gamma + \sigma \mathcal{R}$, 
where $\mathcal{R} \in \mathbb{R}^{k(2k+1) \times k}$ is a random matrix with i.i.d. $N(0, 1)$ entries and  $b_{\sigma}:= A_{\gamma,\sigma} z^* + \delta \te$, where
$\te:=\sum_{j=0}^{k-1}e_{(jk+1)}$ with each $e_{(jk+1)}$ the ${(jk+1)}$th standard unit coordinate vector. 
As in experiment 2, the entries of vector $z^*$ are realizations of 
i.i.d. $N(0, 1)$ random variables. 
For each $\sigma \in [10^{-1}, 10^{-2}, 10^{-3}, 10^{-4}]$, construct $50$ problems with the entries of $\mathcal{R}$ i.i.d. 
 $N(0, 1)$. The DDFG-IRLS algorithm is run on all $50$ problems with each run of the algorithm initialized 
 at a $z^0$ with components selected i.i.d. $N(0, 100)$. 
 The algorithm is terminated when either $\norm{x^k - x^*}_2 < 10^{-3}$ or the number of iterations exceed $10^5$. 
 The results are presented on the right hand side of Figure \ref{cs_counter_2}. 
 Each point with coordinates $(x, y)$ represents the experiment with $\sigma = 10^{-y}$ 
 terminated after $x$ iterations. 
 When $\sigma = 10^{-4}$, the DDFG-IRLS algorithm fails to recover the true $x^*$ within $10^5$ steps for all the $50$ problems. 
 In other words, the failure of DDFG-IRLS is robust to a small random normal perturbation of matrix $A_\gamma$
 and when it does succeed for slightly large perturbations of $A_\gam$ the convergence is still quite slow.
\begin{figure}[H]
  \centering
  \begin{minipage}[b]{0.49\textwidth}
    \includegraphics[width=\textwidth]{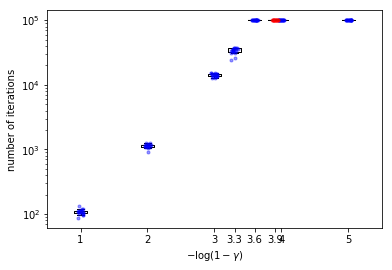}
  \end{minipage}
  \hfill
  \begin{minipage}[b]{0.50\textwidth}
    \includegraphics[width=\textwidth]{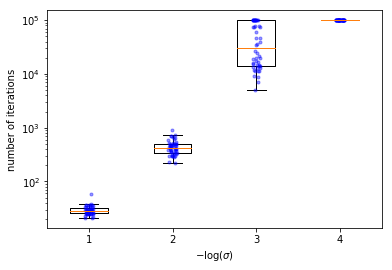}
  \end{minipage}
  \caption{Experiment 2 is presented in the left figure. The red points represents the convergence result for $\gamma = \sqrt{(4k^2(2k+1)^2 + 1)/(4k^2(2k+1)^2 + 4)} = 1-10^{-3.9}$. Experiment 3 is presented in the right figure. In each experiment, every point is a single trial with the corresponding parameter ($\gamma$ and $\sigma$ respectively).}
 \label{cs_counter_2}
\end{figure}
\subsection{Comparison of DDFG-IRLS and Algorithm \ref{algorithm_cs}}\label{sec:num}
In practice the DDFG-IRLS algorithm and Algorithm \ref{algorithm_cs} 
have nearly identical performance on randomly generated problems.
We illustrate this with two additional numerical experiments.

 In experiment 4, the entries of $\Phi \in \mathbb{R}^{300 \times 500}$ are chosen to be i.i.d. $N(0, 1)$ with the solution $x_* \in \mathbb{R}^{500}$ chosen
 so that the first $100$ entries are independent samples from $N(0, 1)$ and the 
 remaining components are taken to be $0$. 
 Set $y := \Phi x_*$. In practice, the NSP parameters $K$ and $\gam$ are not
 known even though they appear explicitly in the updating policy for the smoothing
 parameter $\eps_k$. All that is known is that if the NSP holds,
 then $K< N/2$ and $\gam\in(0,1)$. In this regard, it may be that the 
 DDFG-IRLS algorithm has an edge over Algorithm \ref{algorithm_cs} since the
 performance of Algorithm \ref{algorithm_cs} may be sensitive to the
 choice of $\gamma$. 
Consequently, in this experiment, 
we examine the robustness of the performance of both algorithms
to an ad hoc choice of the NSP parameters $K$ and $\gam$.
For each $K \in \{99, 100, 150, 200, 250, 300\}$ and $\gamma \in \{0.1, 0.5, 0.9\}$, we run Algorithm \ref{algorithm_cs} one hundred times with a
random initialization 
$x_0\sim N(0, 100 \cdot I_5)$ on each run. 
For each of these values of $K$, 
we also run the DDFG-IRLS algorithm for $100$ times with same random initializations 
$x_0\sim N(0, 100 \cdot I_5)$. The results are presented in Figure \ref{experiment_3}. The plot tells us that the success of both
algorithms is robust with respect to the choice of $K$.
When $K$ is strictly smaller than  the true number of the nonzero entries in the solution, 
both algorithms fail regardless of the choice of $\gamma$. 
On the other hand, if we take $K = 250=N/2$ or $K = 300$, both algorithms
succeed.
In addition, the two algorithms have nearly identical performance regardless of
the choice of $K$ when $\gamma$ is chosen to be $0.9$. 
Overall, a degradation in the performance of Algorithm \ref{algorithm_cs} for 
the smaller
values of $\gamma$ only occurs when $K$ is poorly chosen.
In practice, we recommend choosing $K$ be a half of the columns of the measurement matrix $\Phi$ and set $\gamma\approx 0.9$.
In this case, our experiment  indicates that the performance of the two algorithms is essentially identical. 
  \begin{figure}[H]
	\centering
	\includegraphics[width = 0.5\textwidth]{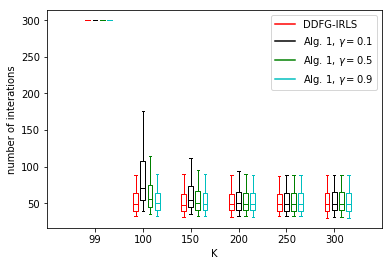}
	\caption{Experiment 4: Each box represent $100$ times runs of the algorithm.} 
	\label{experiment_3}
\end{figure}

In the final numerical experiment 5, we briefly
examine the efficiency of the DDFG-IRLS algorithm
and Algorithm \ref{algorithm_cs} in solving problems with randomly generated 
data. In this experiment we use the fixed parameter setting 
$(K,\gam,\eta)=(N/2,0.9,0.9)$ with $(N,m)=(500,300)$. 
In all of these experiments, the entries of $\Phi$ are 
independent samples from $N(0, 1)$.
In all experiments, the first $k$ entries of $x^*$ are i.i.d. sampled from $N(0, 100)$ with remaining
entries set to zero.
In Figure \ref{pr_evidence1},
$k=100$, $k=120$ in Figure~\ref{pr_evidence2}, and $k=50$ in Figure~\ref{pr_evidence3}.
The experiment is repeated $50$ times for each algorithm. 
The top panel of Figure \ref{pr_evidence1} shows percentage of problems solved versus the number
of iterations, with an iteration max of 12. 
The bottom panel of Figure \ref{pr_evidence} shows the average error
$(1/50)\sum_{i=1}^{50}\text{error}^k_i$ where $\text{error}^k_i$ is the value of $\norm{x^k-x^*}$
in the i$^{\text{th}}$ trial. 
Figures~\ref{pr_evidence2} and~\ref{pr_evidence3}
show the percentage of problems solved versus the number
of iterations for their respective $k$ values
\begin{figure}[H]
    \includegraphics[width=0.49\textwidth]{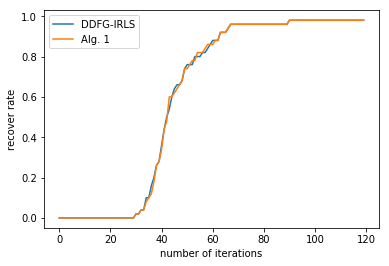}
    \includegraphics[width=0.49\textwidth]{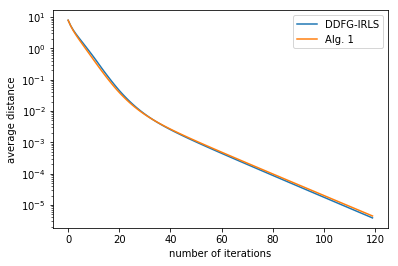}
    \caption{ \label{pr_evidence1} Recovery rate (top panel) and average distance (bottom panel) vs. iterations for $k = 100$. The bottom 
    figure illustrates the linear rate of convergence of the algorithms.}
    \end{figure}
    \begin{figure}[H]
    \includegraphics[width=0.49\textwidth]{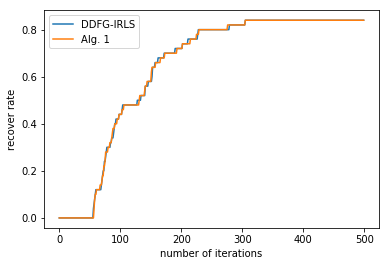}
    \caption{ \label{pr_evidence2} Recovery rate vs. iterations for $k = 50$.}
\end{figure}
\begin{figure}
    \includegraphics[width=0.49\textwidth]{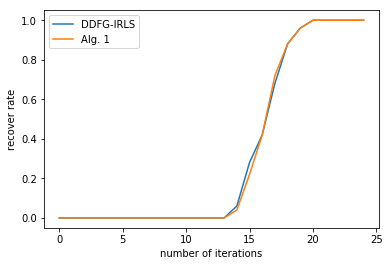}
\caption{ \label{pr_evidence3} Recovery rate vs. iterations for $k = 120$.}
\end{figure}

The DDFG-IRLS and Algorithm \ref{algorithm_cs} perform essentially the same in
these random experiments. The number of iterations required depends on the sparsity
of the solution with the iteration count decreasing with the sparsity $k$. This indicates that these 
algorithms are most useful when the underlying solution is sparse. Finally, the bottom 
panel of Figure \ref{pr_evidence1} demonstrates the linear rate of convergence
of these methods.

\section{Discussion}
In this contribution we provide a concrete example where the DDFG-IRLS fails when $k=K$,
and provide a remedy by changing the updating strategy for the smoothing parameter $\eps_n$.
This remedy increases the range of values for both $k$ and $\gam$ for which
the algorithm provably converges  with a local linear rate to the largest possible intervals 
$[1,K]$ and $(0,1)$ for $k$ and $\gam$, respectively. We have also shown through our numerical experiments
that on randomly generated problems both algorithms are robust to the choice of $K$ and $\gam$
and that their performance is essentially identical. Therefore, if one is concerned about the possible
failure DDFG-IRLS, then Algorithm \ref{algorithm_cs} should be considered with recommended parameter
choices $(K,\eta,\gam)=(N/2,0.9.0.9)$, or equivalently, $0.05\le\eta(1-\gam)\le 0.09$ since knowledge
of the product $\eta(1-\gam)$ is all that is required for implementation.

\bibliographystyle{plain}
\bibliography{cs.bib}

\end{document}